\documentclass[a4paper, draft,11pt]{amsart}
\usepackage[english]{babel}
\usepackage{amsmath}
\usepackage{amssymb}
\usepackage{amscd}
\usepackage{amsthm}
\newtheorem{predl}{Proposition}
\newtheorem{lem}{Lemma}
\newtheorem{theorem}{Theorem}
\newtheorem*{sld}{Corollary}

\newcommand*{\hm}[1]{#1\nobreak\discretionary{}{\hbox{$\mathsurround=0pt #1$}}{}}
\DeclareMathOperator{\cone}{cone } \DeclareMathOperator{\Pic}{Pic}
\DeclareMathOperator{\spec}{Spec }  \textwidth=16.3cm
\begin{document}
\date{}
\title{GIT-equivalence and diagonal actions}
\author[P. Yu. Kotenkova]{Polina Yu. Kotenkova}
\address{Department of Higher algebra, Faculty of Mechanics and Mathematics, Moscow State Lomonosov University, Leninskie Gory~1, Moscow, 119991, Russia}
\email{kotpy@mail.ru} \maketitle

\begin{abstract}
We describe the GIT-equivalence classes of linearized ample line
bundles for the diagonal actions of the linear algebraic groups
$\text{SL}(V)$ and $\text{SO}(V)$ on ${\mathbb{P}(V)^{m_1}\times
\mathbb{P}(V^*)^{m_2}}$ and $\mathbb{P}(V)^m$ respectively.
\end{abstract}

\section{Introduction}

Let $G$ be a complex reductive algebraic group, $X$ a projective
$G$-variety, and $L$ an ample $G$-linearized line bundle over the
variety $X$. Classical Mumford's construction, see \cite{4}, takes
these objects to the open subset of semi-stable points
$$X_L^{ss}=\{x \in X : \, F(x) \ne 0 \,\, \mbox{for some}
\,\, m>0 \,\, \mbox{and} \,\, F \in \Gamma(X, L^{\otimes m})^G\}$$
and the categorical quotient  $X_L^{ss} \longrightarrow
X_L^{ss}/\!/G$. Two $G$-linearized line bundles $L_1$ and $L_2$
over the $G$-variety $X$ are called {\itshape GIT-equivalent} if
$X_{L_1}^{ss}=X_{L_2}^{ss}$. The papers \cite{3}, \cite{5} and
\cite{6} are devoted to the study of GIT-equivalence. It is shown
that the GIT-equivalence classes define the fan structure on the
cone of $G$-linearized ample line bundles. The main approach used
to describe the cones in this fan is the Hilbert-Mumford criterion
\cite[Chapter 2]{4}; also see \cite[Example 3.3.24]{3}, where the
GIT-equivalence classes are described for the diagonal action of
the group $SL(V)$ on the variety многообразии $\mathbb{P}(V)^m$.

In \cite{2} an elementary description of GIT-equivalence classes
for algebraic torus actions is obtained. The authors use so called
orbit cones. Using the Cox construction, in \cite{1} this
description is adopted for a large class of $G$-varieties, compare
\cite[Section 3]{6}. In \cite[Theorem 6.2]{1} there is also a
description of the GIT-equivalence classes for the diagonal action
of  the symplectic group $Sp(V)$ on $\mathbb{P}(V)^m$.

The aim of this paper is to find the GIT-equivalence classes for
the diagonal actions of other classical groups. In Section $2$ we
give some necessary information from works \cite{1} and \cite{2}.
In Section~$3$ we describe the GIT-fan  for the diagonal action of
the special orthogonal group $SO(V)$, and in Section~$4$ for the
diagonal action of the group $SL(V)$ on the variety
${\mathbb{P}(V)^{m_1}\times \mathbb{P}(V^*)^{m_2}}$. These results
are based on the description of generators of the algebra  of
invariants (The First Fundamental Theorem of the Classical
Invariant Theory).

The author is grateful to I.V. Arzhantsev for posing the problem
and his permanent support.

\section{Orbit cones and the GIT-fan}

Let $G\subseteq GL(V)$ be a complex algebraic group acting
diagonally on the space ${\mathbb{V}=V^{m_1}\hm\oplus
(V^*)^{m_2}}$, $m_1\hm+m_2=m$. Denote by $P(a_1,\ldots,a_m)\subset
\mathbb{C}[\mathbb{V}]$, ${a_i\in \mathbb Z_{\geqslant 0}}$, the
subspace consisting of homogeneous polynomials of multidegree
$(a_1,\ldots,a_m)$. To each vector ${a=(a_1,\ldots,a_m) \in
\mathbb Z_{\geqslant 0}^m}$ assign the open subset
$$U(a)\hm=\{v\in \mathbb{V} \mid \exists k \in \mathbb N, \,\,F
\in P(ka_1,\ldots,ka_m)^G : F(v)\neq 0\},$$ where
$P(ka_1,\ldots,ka_m)^G \subseteq P(ka_1,\ldots,ka_m)$ is the
subspace of $G$-invariants. Note that the subset $U(a)$
corresponds to the set of semistabe points $X_L^{ss}$, where
${X\hm=\mathbb{P}(V)^{m_1}\times \mathbb{P}(V^*)^{m_2}}$. Here the
line bundle $L$ is represented by the point
$a\in\mathbb{Z}^{m}\cong\Pic(X)$.

Two points $a$ and $b \in \mathbb Z_{\geqslant0}^m$ are called
{\itshape GIT-equivalent} if ${U(a)=U(b)}$.

Assume that the algebra of invariants $\mathbb{C}[\mathbb{V}]^G$
is finitely generated and $F_1, \ldots, F_r$ are its homogeneous
generators. Denote by $a(1), \ldots, a(r) \in
\mathbb{Z}^m_{\geqslant 0}$ the multidegrees of $F_1, \ldots,
F_r$.

{\itshape The weight cone} $\Omega \subset \mathbb{Q}^m$ is the
cone generated by $a(1), \ldots, a(r)$.

\begin{lem}
The set $U(a)$ is non-empty if and only if $a \in \Omega$.
\end{lem}

\begin{proof}
Suppose $U(a) \neq \varnothing$. Then there exist $k\in \mathbb N$
and $F \in P(ka_1,\ldots,ka_m)^G$ such that $F\not\equiv 0$. Since
$F \in \mathbb{C}[F_1,\ldots,F_r]$, we have
$$F\hm=\sum\limits_{p_1,\ldots,p_r} c_{p_1\ldots p_r}F_1^{p_1}
\ldots F_r^{p_r}.$$ If $c_{p_1\ldots p_r}\ne 0$, then we obtain
$ka\hm=p_1a(1)\hm+\ldots\hm+p_ra(r)$. Hence $a\in \Omega$.

Conversely, assume that $a\in \Omega$. Then
$$a\hm=\lambda_1a(1)\hm+\ldots\hm+\lambda_ra(r),$$ where
$\lambda_1,\ldots,\lambda_r  \in \mathbb Q_{\geqslant 0}$.
Multiplying this equation by the common denominator of
$\lambda_1,\ldots,\lambda_r$, we get
$$ka\hm=c_1a(1)\hm+\ldots\hm+c_ra(r),$$ where $k,c_1,\ldots,c_r \in
\mathbb Z_{\geqslant0}$. Then $F_1^{c_1}\ldots F_r^{c_r} \in
P(ka)^G$, $F_1^{c_1} \ldots F_r^{c_r} \not\equiv 0$. Hence $U(a)
\neq \varnothing$.
\end{proof}

Let $v \in \mathbb{V} $. {\itshape The orbit cone} associated to
$v$ is the rational cone $$\omega(v)=\cone(a \mid \exists F \in
P(a)^G : F(v) \neq 0).$$

\begin{lem} \label{l1}
One has $\omega(v)=\cone(a(i) \mid F_i(v) \neq 0)$.
\end{lem}

\begin{proof}
It is evident that $\cone(a(i) \mid F_i(v)\ne 0)$ is contained in
$\omega(v)$.

Consider the point $a \in \mathbb Z^m_{\geqslant0}$ and the
polynomial
$$F\hm=\sum\limits_{p_1,\ldots,p_r} c_{p_1\ldots
p_r}F_1^{p_1}\hm \ldots F_r^{p_r} \hm\in P(a)^G,$$ such that $F(v)
\ne 0$. Then there is a summand $c_{p_1\ldots p_r}F_1^{p_1} \ldots
F_r^{p_r}$ not vanishing at $v$. Further, if $p_i \ne 0$, then
$F_i(v) \ne 0$. Therefore $a\hm=p_{i_1}a(i_1)\hm+\ldots\hm
+p_{i_s}a(i_s)$, where $F_{i_l}(v) \ne 0, \, l=1,\ldots, s$. Hence
we obtain the inverse inclusion $\cone(a(i) \mid F_i(v)\ne
0)\supseteq \omega(v)$.
\end{proof}
\begin{sld}
 The set of cones $\{\omega(v) \mid v \in \mathbb{V}\}$ is  finite.
\end{sld}

\begin{predl} \label{pr1}
Two points $a$ and $b$ are GIT-equivalent if and only if for any
$v \in \mathbb{V}$ either $a \in \omega(v)$ and $b \in \omega(v)$,
or $a \notin \omega(v)$ and $b \notin \omega(v)$.
\end{predl}
\begin{proof}
Suppose $U(a)=U(b)$ and $a \in \omega(v)$. It follows from Lemma
\ref{l1} that
$$a\hm=\lambda_{i_1}a(i_1)\hm+\ldots\hm+\lambda_{i_s}a(i_s),$$
where $F_{i_j}(v)\ne 0, \, j=1,\ldots,s$, $\lambda_j \in \mathbb
Q_{\geqslant 0}$. Multiplying this equation by the common
denominator of $\lambda_1,\ldots,\lambda_r$, we get
$$ka\hm=p_1a(i_1)\hm+\ldots\hm+p_sa(i_s),$$ where $k,p_1,\ldots,p_s
\in \mathbb Z_{\geqslant 0}$. Then $F_{i_1}^{p_1}\ldots
F_{i_s}^{p_s} \in P(ka)^G$ does not vanish at $v$, and hence ${v
\in~ U(a)\hm=U(b)}$. Therefore there exist $l \in \mathbb N$ and
$F \in P(lb)^G$ such that $F(v) \ne 0$. Thus $lb \in \omega(v)$
and $b \in \omega(v)$. Similarly if $b \in \omega(v)$, then $a \in
\omega(v)$.

It can easily be checked that $a \in \omega(v)$ if and only if $v
\in U(a)$. If for any $v \in \mathbb{V}$ either $a \in \omega(v)$
and $b \in \omega(v)$, or $a \notin \omega(v)$ and $b \notin
\omega(v)$ hold, then for any  $v \in \mathbb{V}$ we have either
$v \in U(a)$ and $v \in U(b)$, or $v \notin U(a)$ and $v \notin
U(b)$. Hence $U(a)=U(b)$.
\end{proof}

{\itshape The GIT-cone} of a point $a \in \mathbb Z_{\geqslant
0}^m$ is the cone $\tau(a)\hm= \bigcap\limits_{a \in \omega(v)}
\omega(v)$.

Recall that a finite set $\{\sigma_i\}$ of cones in $\mathbb Q^m$
is called {\itshape a fan}, if

$(1)$ each face of a cone in $\{\sigma_i\}$ is also a cone in
$\{\sigma_i\}$;

$(2)$ the intersection of two cones in $\{\sigma_i\}$ is a face of
each.

\pagebreak
\begin{theorem}
The set of the cones $\Psi=\{\tau(a)\,|\, a \in \Omega\}$ is a
fan.
\end{theorem}

The proof may be found in \cite[Theorem 2.11]{2}.

The fan $\Psi$ is called the {\itshape GIT-fan}. It follows from
Proposition~\ref{pr1} that the classes of GIT-equivalence are
relative interiors of GIT-cones.

\vspace{0.5cm} Let $T=(\mathbb {C}^\times)^m$ be a torus. It acts
on the space
 $\mathbb{V}$ as $$t\circ
(v_1,\ldots,v_{m_1},l_1,\ldots,l_{m_2})=(t_1v_1,\ldots,t_{m_1}v_{m_1},s_1l_1,\ldots,s_{m_2}l_{m_2}),$$
where $t=(t_1,\ldots,t_{m_1},s_1,\ldots,s_{m_2})\in T$,
$(v_1,\ldots,v_{m_1},l_1,\ldots,l_{m_2})\in \mathbb{V}$. This
action commutes with the action of $G$, hence the action of $T$ on
the categorical quotient $\mathbb{V}/\!/G:=\spec {\mathbb
C}[{\mathbb V}]^G$ is well defined.

Consider a point $v \in \mathbb{V}$. It is not hard to see that
$$\dim
\omega(v)+\dim T_{\pi(v)}=\dim T,$$ where ${\pi: \mathbb{V}
\longrightarrow \mathbb{V}/\!/G}$ is the quotient morphism, and
$T_{\pi(v)}$ is the stabilizer of the point $\pi(v)$.

Note that our definition of the orbit cone agrees with
\cite[Definition 2.1]{2} for the action of the torus $T$ on the
variety $\mathbb{V}/\!/G$.

\section{The case of $SO(V)$}

Consider $G=SO(V)$, $\dim V \geqslant 3$. Let $(.,.)$ be a
non-degenerate symmetric bilinear form on $V$ preserved by
$SO(V)$. Since $V$ is $G$-isomorphic to its dual $V^*$, we can
assume that $m_2=0,\, m=m_1$, and $\mathbb{V}=V^m$. Let us
construct the GIT-fan for the diagonal action $SO(V)$ on the
variety $\mathbb{P}(V)^m$.

The algebra of invariants for the action $SO(V)$ on $\mathbb{V}$
is generated by ${u_{ij}=(v_i,v_j)}$, where $(v_1,\ldots,v_m) \in
V^{m}$ \cite[$\S$ 9.3]{0}. Here the multidegrees are
$f_{ii}:=(0,\ldots,0,\underbrace {2}_{i},0,\ldots,0)$ and
$f_{ij}:=(0,\ldots,0,\underbrace1_i,0,\ldots,0,\underbrace1_j,0,\ldots,0)$,
${i,j=1,\ldots,m}$. The morphism $\pi: \mathbb{V} \longrightarrow
\mathbb{V}/\!/G$ sends $(v_1,\ldots,v_m)$ to the symmetric matrix
$\left( (v_i,v_j) \right) _{i,j=1}^m$.

The GIT-fan is contained in $\mathbb {Q}^{m}$. Let
$x_1,\ldots,x_m$ be the coordinates in this space.

Clearly, the weight cone $\Omega$ generated by $f_{ij}$ is given
by inequalities
$$x_i\geqslant0,\, i=1,\ldots,m.$$

\begin{predl} \label{pr2} Each $(m-1)$-dimensional orbit cone lies
in some hyperplane
\begin{equation}
\sum_{i \in I} x_i=\sum_{j \in J} x_j, \label{f1}
\end{equation} where
$I, J \subset \{1,\ldots, m\}$, $I \ne \varnothing$, $J \ne
\varnothing$, $I\cap J = \varnothing$.
\end{predl}

\begin{proof}
The torus $T=(\mathbb {C}^\times)^m$ acts on $V^m$: ${t\circ
(v_1,\ldots,v_m)}=(t_1v_1,\ldots,t_mv_m).$ Then ${t\circ (v_i,
v_j)\hm=t_it_j(v_i,v_j)}$.  The orbit cone associated to
$v\hm=(v_1,\ldots,v_m)$ is $(m-1)$-dimensional if and only if the
stabilizer $T_{\pi(v)}$ of the point $\pi(v)$ is one-dimensional.

Consider a graph $\Gamma_v$ with the set of vertices
$\{v_1,\ldots,v_m\}$. By definition, $v_i$ and $v_j $ are joined
by an edge in $\Gamma_v$ if and only if $(v_i,v_j)\neq 0$. Assume
that $(v_i,v_j)\neq 0$. Then any $t\in T_{\pi(v)}$ satisfies
$t_i=t_{j}^{-1}$.

Let $\Gamma_v=\Gamma_1\sqcup\ldots\sqcup\Gamma_l$ be the
decomposition into connected components. If $\Gamma_k$ contains a
cycle of odd length or a loop (type A), then $t_i^2=1$ for all
$v_i \in \Gamma_k$ and $t\in T_{\pi(v)}$. In other case (type B),
it is possible to divide the set of vertices of $\Gamma_k$ into
two subsets. For a point of the first subset $t_i=s_k$ holds, and
for a point of the second subset we have $t_i=(s_k)^{-1}$, where
$s_k\in \mathbb{C}^\times$. The stabilizer is one-dimensional if
and only if there is only one component of type B in the
graph~$\Gamma_v$. Denote by $I$ and $J$ the sets of vertices in
the first and the second subsets of this component. The weight
$f_{ij}$ lies in $\omega(v)$ if and only if $i \in I, j \in J$ or
$j \in I, i \in J$. Hence the orbit cone is contained in
hyperplane~\eqref{f1}.
\end{proof}

It follows from Proposition \ref{pr2} that if dimension of the
orbit cone $\omega(v)$ is less than $m-1$, then $\omega(v)$ lies
in the intersection of some $(m-1)$-dimensional orbit cones. Thus
two points are GIT-equivalent if and only if they lie in the same
$m$-dimensional and $(m-1)$-dimensional orbit cones.

\begin{theorem} \label{t2} For the diagonal action of the group $SO(V)$ on the variety
$\mathbb{P}(V)^m$ the GIT-fan is obtained by cutting of the cone
$$\Omega=\{(x_1,\ldots,x_m)\, | \, x_i\geqslant0\}$$ by hyperplanes
$$\sum_{i \in I} x_i=\sum_{j
\in J} x_j, \leqno(\ref{f1})$$ where $I, J \subset \{1,\ldots,
m\}$, $I \ne \varnothing$, $J \ne \varnothing$, $I\cap J =
\varnothing$.
\end{theorem}
\begin{proof} It is sufficient to find all $(m-1)$-dimensional orbit
cones. It follows from Proposition \ref{pr2}  that we should only
prove that the intersection of each hyperplane \eqref{f1} with the
cone $\Omega$ is the orbit cone for some point $v$.

Let $v_k=(1,\imath,0,\ldots,0)$ for $k \in I$,
$v_j=(1,-\imath,0,\ldots,0)$ for $j \in J$, and
$v_l=(0,0,1,0,\ldots,0)$ for $l \not\in I\cup J$. (Here
${\imath^2=-1}$). The orbit cone associated to $v$ is generated by
the weights $f_{kj} (k\in I,\,j \in J)$ and $f_{ll} (l \not\in
I\cup J)$. Hence $\omega(v)$ is $(m-1)$-dimensional and lies in
hyperplane \eqref{f1}. It is easy to check that the rays $\langle
f_{kj}\rangle (k\in I,\,j \in J)$ and $\langle f_{ll}\rangle (l
\not\in I\cup J)$ are precisely the edges of the intersection of
$\Omega$ with hyperplane \eqref{f1}. This completes the proof of
Theorem \ref{t2}.

\end{proof}

{\bfseries Example.} Consider the action of $\text{SO}_3$ on the
space ${\mathbb{C}^3 \oplus \mathbb{C}^3 \oplus \mathbb{C}^3}.$

The weight cone is the cone $\Omega=\{(x_1,x_2,x_3) \in
\mathbb{Q}^3\,|\, x_1\geqslant 0,x_2\geqslant 0,x_3\geqslant 0\}$.
The GIT-fan is obtained by cutting of the cone $\Omega$ by
hyperplanes
$$x_1=x_2,\,\, x_1=x_3,\,\, x_2=x_3,$$
$$x_1+x_2=x_3,\,\, x_1+x_3=x_2,\,\, x_2+x_3=x_1.$$

The intersection of the GIT-fan with the hyperplane
${x_1+x_2+x_3=1}$ looks like:

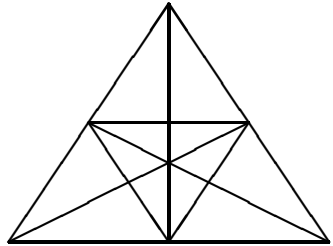
\begin{center}
\begin{picture}(120,100)
\thicklines \put(0,0){\line(1,0){120}} \put(0,0){\line(2,3){60}}
\put(120,0){\line(-2,3){60}} \put(60,0){\line(0,1){90}}
\put(30,45){\line(1,0){60}} \put(30,45){\line(2,-3){30}}
\put(90,45){\line(-2,-3){30}} \put(0,0){\line(2,1){90}}
\put(120,0){\line(-2,1){90}}
\end{picture}
\end{center}
\hspace{0.5cm}

There are $33$ classes of GIT-equivalence: $12$ classes are
three-dimensional, $21$ classes are two-dimensional, and $10$
classes are one-dimensional.

\section{The case of $SL(V)$}

Consider $G=SL(V)$. Let us construct the GIT-fan for the diagonal
action $SL(V)$ on the variety ${\mathbb{P}(V)^{m_1}\times
\mathbb{P}(V^*)^{m_2}}$.

The algebra of invariants $\mathbb {C}[\mathbb{V}]^{SL(V)}$, where
$\mathbb{V}=V^{m_1}\hm\oplus (V^*)^{m_2}$, is generated by
$\det(v_{i_1},\ldots,v_{i_n})$, $\det(l_{j_1},\ldots,l_{j_n})$,
and $l_j(v_i)$, where $(v_1,\ldots,v_{m_1},l_1, \ldots, l_{m_2})
\in V^{m_1} \oplus (V^*)^{m_2}$ \cite[$\S$ 9.3]{0}. Here the
multidegrees are $$f_{i_1\ldots i_n}=(\alpha^1_{i_1\ldots
i_n},\ldots,\alpha^{m_1}_{i_1\ldots i_n},\underbrace {0, \ldots
,0}_{m_2}),\,\, g_{j_1\ldots
j_n}=(\underbrace{0,\ldots,0}_{m_1},\beta^1_{j_1,\ldots
j_n},\ldots,\beta^{m_2}_{j_1\ldots j_n}),$$ $$\mbox{and}\,\,\,
h_{ij}=(\varepsilon^1_{ij},\ldots,\varepsilon^{m_1}_{ij},\delta^1_{ij},\ldots,\delta^{m_2}_{ij}),$$
where $$\alpha^{i_1}_{i_1\ldots
i_n}\hm=\ldots=\alpha^{i_n}_{i_1\ldots
i_n}\hm=\beta^{j_1}_{j_1\ldots
j_n}\hm=\ldots\hm=\beta^{j_n}_{j_1\ldots
j_n}\hm=\varepsilon^i_{ij}\hm=\delta^j_{ij}\hm=1,$$ and other
numbers are zero.

The GIT-fan is contained in $\mathbb {Q}^{m}$. Let
$x_1,\ldots,x_{m_1},$ $y_1,\ldots,y_{m_2}$ be the coordinates in
this space.

First, suppose that $m_1\geqslant n$ or $m_2\geqslant n$.

\begin{predl} Each $(m-1)$-dimensional orbit cone lies
in one of the hyperplanes
\begin{gather} x_i=0, \,  i=1,\ldots,m_1, \label{f2}\\ y_j=0, \,
j=1,\ldots,m_2, \label{f3} \\
x_1+\ldots+x_{m_1}=y_1+\ldots+y_{m_2}, \label{f4}\\
(n-k)\sum_{i \in I}x_i-k\sum_{i \not \in I}x_i=(n-k)\sum_{j \in
J}y_j-k\sum_{j \not \in J}y_j, \label{f5} \end{gather} where
$1\leqslant k \leqslant n-1$, $I \subset \{1,\ldots,m_1\},$ $J
\subset \{1,\ldots,m_2\}$, and either $k\leqslant |I| \leqslant
m_1-n+k$ or ${k\leqslant |J| \leqslant m_2-n+k}$.
\end{predl}

\begin{proof}
If there is a zero vector or a zero function in the set
$\{v_1,\ldots,v_{m_1},l_1,\ldots,l_{m_2}\}$, then the orbit cone
associated to $v=(v_1,\ldots,v_{m_1},l_1,\ldots,l_{m_2})$ lies in
hyperplane of type \eqref{f2} or \eqref{f3}. Further, we assume
that all the components of $v$ are nonzero.

The torus $T=(\mathbb {C}^\times)^{m}$ acts on $V^{m_1} \oplus
(V^*)^{m_2}$ as above. Then for any $t\in T$ we have $$t\circ
l_j(v_i)=t_is_jl_j(v_i),$$
$$t\circ\det(v_{i_1},\ldots,v_{i_n})=t_{i_1}\ldots
t_{i_n}\det(v_{i_1},\ldots,v_{i_n}),$$
$$t\circ\det(l_{j_1},\ldots,l_{j_n})=s_{j_1}\ldots
s_{j_n}\det(l_{j_1},\ldots,l_{j_n}),$$ and the orbit cone
associated to $v$ is $(m-1)$-dimensional if and only if the
stabilizer $T_{\pi(v)}$ of the point $\pi(v)$ is one-dimensional.

Consider a graph $\Gamma_v$ with
$\{v_1,\ldots,v_m,l_1,\ldots,l_{m_2}\}$ as the set of vertices. By
definition, $v_i$ and $l_j $ are joined by an edge in $\Gamma_v$
if and only if $l_j(v_i)\neq 0$. If vertices
$v_{i_1},v_{i_2},l_{j_1},l_{j_2}$ lie in the same connected
component of the graph~$\Gamma_v$, then any $t\in T_{\pi(v)}$
satisfies $t_{i_1}=t_{i_2}$, $s_{j_1}\hm=s_{j_2}$,
$t_{i_1}=s_{j_1}^{-1}$.

\vspace{0.15cm}{\bfseries Case 1:} $\dim\langle
v_1,\ldots,v_{m_1}\rangle < n$, $\dim\langle
l_1,\ldots,l_{m_2}\rangle < n$.

In this case all the determinants are zero. If the stabilizer
$T_{\pi(v)}$ is one-dimensional, then the graph $\Gamma_v$ is
connected. The orbit cone is generated by the weights
$\{h_{ij}\}$. Their span is $(m-1)$-dimensional and lies in the
hyperplane $x_1\hm+\ldots\hm+x_{m_1}=y_1\hm+\ldots\hm+y_{m_2}$.
Hence the orbit cone lies in hyperplane of type~\eqref{f4}.

\vspace{0.15cm} {\bfseries Case 2:} $\dim\langle
v_1,\ldots,v_{m_1}\rangle = n$, $\dim\langle
l_1,\ldots,l_{m_2}\rangle = n$.

Suppose $\det(v_{i_1}, \ldots, v_{i_n}) \ne 0$. Then
$v_{i_1},\ldots,v_{i_n}$ is a basis of $V$. For any  $t\in
T_{\pi(v)}$ the equation $t_{i_1}\ldots t_{i_n}=1$ holds. If
$v_{i_1}$ occurs in the decomposition of $v_i$ with respect to the
basis $v_{i_1},\ldots,v_{i_n}$, then $\det(v_i, v_{i_2},\ldots,
v_{i_n})\ne 0$ and $t_it_{i_2}\ldots
t_{i_n}\hm=t_{i_1}t_{i_2}\ldots t_{i_n}\hm=1$. Hence
$t_i=t_{i_1}$. Similarly consider other $v_{i_j}, \mbox{where}
\,\, j\hm=2,\ldots,n$. Thus the space $V$ decomposes into the sum
$V\hm=V_{k_1}\hm\oplus\ldots\hm\oplus V_{k_r}$. The torus $T$
multiples any $V_{k_l}$ by $\overline{t}_l$
($\overline{t}_l=t_{i_j}$ for some $j$). In the same way $V^*$
decomposes into the sum
$V^*\hm=W_{\widetilde{k}_1}\hm\oplus\ldots\hm\oplus
W_{\widetilde{k}_q}$, the torus $T$ acts on any $W_{k_l}$ as
multiplication by $\overline{s}_{l}$. Thus any element $t\in
T_{\pi(v)}$ satisfies the conditions  $(\overline{t}_1)^{\dim
V_{k_1}}\ldots(\overline{t}_r)^{\dim V_{k_r}}=1$ and
$(\overline{s}_1)^{\dim W_{\widetilde{k}_1}}\ldots
(\overline{s}_q)^{\dim W_{\widetilde{k}_q}}=1$.

Consider a new graph $\Gamma'_v$ with $V_{k_1},\ldots,V_{k_r}$,
$W_{\widetilde{k}_1},\ldots,W_{\widetilde{k}_q}$ as the set of
vertices. The vertices  $V_k$ and $W_{\widetilde{k}}$ are joined
by an edge in $\Gamma'_v$  if and only if there exist $v_i \in
V_k$ and $l_j \in~W_{\widetilde{k}}$ such that $l_j(v_i) \ne 0$.

Denote by $H_1,\ldots,H_p$ the connected components of the graph
$\Gamma'_v$. Let
$$V'_i\hm=\bigoplus\limits_{V_k \in \, H_i} V_k,\,\,\,
W'_i\hm=\bigoplus\limits_{W_{\widetilde{k}} \in \, H_i}
W_{\widetilde{k}}.$$ Then $T_{\pi(v)}$ multiples $V'_i$ by $t'_i$
and $W'_i$ by $(t'_i)^{-1}$. The stabilizer is given by the
equations
$$(t'_1)^{\dim V'_1}\ldots(t'_p)^{\dim V'_p}=1,$$
$$(t'_1)^{\dim W'_1}\ldots(t'_p)^{\dim W'_p}=1.$$

By construction of $H_i$, all linear functions of $W'_i$ vanish at
all vectors of $V'_j$ for $i \ne j$, hence ${\dim W'_i\leqslant
\dim V'_i}$. But $\sum\limits_{i=1}^p \dim
W'_i\hm=\sum\limits_{i=1}^p \dim V'_i$, hence ${\dim W'_i\hm=\dim
V'_i}$.

Thus the stabilizer is given by the equation
$$(t'_1)^{\dim V'_1}\ldots(t'_p)^{\dim V'_p}=1,$$ and is
one-dimensional if and only if $p=2$.

So if the orbit cone associated to
$v=(v_1,\ldots,v_{m_1},l_1,\ldots,l_{m_2})$ is
$(m-1)$-dimensional, then $V\hm=V_1\oplus V_2$, $V^*=W_1\oplus
W_2$, ${\dim V_1=\dim W_1=k}$, $1\leqslant k\leqslant n-1$; any
vector $v_i$ lies in $V_1$ or in $V_2$; any linear function $l_j$
lies in $W_1$ or in $W_2$; any linear function from $W_j$ is zero
on any vector from $V_i$ for $i \ne j$.

Let $I$ be the set of numbers of vectors $v_i$ from $V_1$, $J$ be
the set of numbers of linear functions  $l_j$ from $W_1$. Then the
orbit cone associated to $v$ lies in hyperplane given by the
equation
$$(n-k)\sum_{i \in I}x_i-k\sum_{i \not \in I}x_i=(n-k)\sum_{j \in
J}y_j-k\sum_{j \not \in J}y_j.$$ Here the inequalities
$k\leqslant |I| \leqslant m_1-n+k$ and $k\leqslant |J| \leqslant
m_2-n+k$ are satisfied.

\vspace{0.15cm}{\bfseries Case 3:} $\dim\langle
v_1,\ldots,v_{m_1}\rangle = n$, $\dim\langle
l_1,\ldots,l_{m_2}\rangle < n$ or $\dim\langle
v_1,\ldots,v_{m_1}\rangle < n$, ${\dim\langle
l_1,\ldots,l_{m_2}\rangle \hm= n}$. In this case we have one
equation on the stabilizer and the graph $\Gamma_v$ should have
two connected components. We obtain hyperplanes given by equations
\eqref{f5}.
\end{proof}

\begin{predl}\label{pr3} The weight cone $\Omega$ is given by
inequalities
\begin{gather}
 x_l\geqslant 0 ,\,\,\,l=1,\ldots,m_1,\,\,\, y_p\geqslant 0,\,\,\, p=1,\ldots,m_2, \label{f6}\\
 (n-k)(\sum_{j=1}^{m_2}
 y_j-\sum_{i \in I}x_i)+k\sum_{i \not \in I}x_i\geqslant0, \,\,\, (n-k)(\sum_{i=1}^{m_1}x_i-\sum_{j \in
J}y_j)+k\sum_{j \not \in J}y_j\geqslant 0 \notag,
\end{gather}
where $1\leqslant k \leqslant n-1$, $I \subset \{1,\ldots,m_1\},$
$J \subset \{1,\ldots,m_2\}$, $|I|=|J|=k$.
\end{predl}

\begin{proof}
It is sufficient to find hyperplanes \eqref{f2}--\eqref{f5} which
contain facets of the cone $\Omega$. It is clear that hyperplanes
\eqref{f2} and \eqref{f3} do. Since the weights $\{f_{i_1\ldots
i_n}\}$ and $\{g_{j_1\ldots j_n}\}$ lie on different sides of
hyperplane \eqref{f4}, this hyperplane intersects the interior of
the cone $\Omega$.

Consider equation \eqref{f5}. First suppose that ${0<|I|<m_1}$,
${0<|J|<m_2}$. Let $i_1 \in I$, $i_2 \not\in I$, $j_1 \in J$, $j_2
\not\in~J$. The weights $h_{i_1j_2}$ and $h_{i_2j_1}$ lie on
different sides from the hyperplane. Now let $|J| = m_2$. If
$|I|>k$, then there exist numbers $i_1, \ldots, i_{k+1} \in I$,
$i_{k+2},\ldots, i_n \not\in I$. Weights $f_{i_1\ldots i_n}$ and
$g_{j_1\ldots j_n}$ lie on different sides from the hyperplane.
For $|I|=k$ we obtain  inequalities \eqref{f6}. The cases $|I| =
0,m_1$ and $|J| = 0$ are analyzed similarly.
\end{proof}

It follows from the proof of Proposition \ref{pr3} that
hyperplanes \eqref{f4} and \eqref{f5} intersect the interior of
the cone $\Omega$. Further, if dimension of the orbit cone
$\omega(v)$ is less than $m-1$, then $\omega(v)$ lies in the
intersection of some $(m-1)$-dimensional orbit cones. Thus two
points are GIT-equivalent if and only if they lie in the same
$m$-dimensional and $(m-1)$-dimensional orbit cones.

\pagebreak
\begin{theorem} \label{t3} For the diagonal action of the group $SL(V)$ on the variety
$\mathbb{P}(V)^{m_1}\times \mathbb{P}(V^*)^{m_2}$, where $m_1 \,\,
\mbox{or}\,\,\, m_2$ does not exceed  $n=\dim V$, the GIT-fan is
obtained by cutting of the cone $\Omega$ given by inequalities
\eqref{f6} by hyperplanes
$$x_1+\ldots+x_{m_1}=y_1+\ldots+y_{m_2}, \leqno (\ref{f4})$$
$$(n-k)\sum_{i \in I}x_i-k\sum_{i \not \in I}x_i=(n-k)\sum_{j \in J}y_j-k\sum_{j \not \in
J}y_j, \leqno (\ref{f5})$$ where $1\leqslant k \leqslant n-1$, $I
\subset \{1,\ldots,m_1\},$ $J \subset \{1,\ldots,m_2\},$
$k\leqslant |I| \leqslant m_1-n+k$ or ${k\leqslant |J| \leqslant
m_2-n+k}$.
\end{theorem}

\begin{proof}
Let us find all $(m-1)$-dimensional orbit cones. It follows from
Proposition \ref{pr3}  that we should only prove that the
intersection $\Pi$ of any hyperplane of type \eqref{f4} or
\eqref{f5} with the cone $\Omega$ is the orbit cone for some point
$v$.

For the case of hyperplane \eqref{f4} let
$v\hm=(e_1,\ldots,e_1,e^1,\ldots,e^1)$. The orbit cone associated
to $v$ lies in this hyperplane and its dimension equals $m-1$.
Note that the inequalities
\begin{gather*} (n-k)(\sum_{j=1}^{m_2}
 y_j-\sum_{i \in I}x_i)+k\sum_{i \not \in I}x_i\geqslant0\,\,\, \mbox{и} \\  (n-k)(\sum_{i=1}^{m_1}x_i-\sum_{j \in
J}y_j)+k\sum_{j \not \in J}y_j\geqslant 0,  \end{gather*} for the
points of hyperplane \eqref{f4} become
$$\sum_{i \not \in I}x_i\geqslant0 \,\,\, \mbox{и} \,\,\, \sum_{j
\not \in J}y_j\geqslant0.$$ Hence the cone $\Pi$ is the
intersection of hyperplane \eqref{f4} with the positive ortant.
Finally, the weights which generate the orbit cone $\omega(v)$ and
edges of the cone $\Pi$ lie on the same rays.

Denote by $H$ the hyperplane \eqref{f5}. Without loss of
generality it can be assumed that equation~\eqref{f5} is of the
form
\begin{equation*}(n-k)\sum_{i=1}^{|I|}
x_i-k\sum_{i=|I|+1}^{m_1}x_i=(n-k)\sum_{j=1}^{|J|}
y_j-k\sum_{j=|J|+1}^{m_2}y_j, \end{equation*} where $k\leqslant
|I| \leqslant m_1-n+k$.

The orbit cone $\omega(v)$ associated to the point
\begin{multline*}v=(e_1,\ldots,e_k,\underbrace{e_k,\ldots,e_k}_{|I|-k},e_{k+1},\ldots,e_n,\underbrace{e_n,\ldots,e_n}_{m_1+k-|I|-n},
\\ \underbrace{e^1+\ldots+e^k,\ldots,
e^1+\ldots+e^k}_{|J|},\underbrace{e^{k+1}+\ldots+e^n,\ldots,
e^{k+1}+\ldots+e^n}_{m_2-|J|})\end{multline*} is
$(m-1)$-dimensional and lies in hyperplane $H$. It is easy to
check that this cone lies in $\Pi$. It remains to prove the
converse implication.

Let $A=(x_1,\ldots,x_{m_1},y_1,\ldots,y_{m_2}) \in \Pi$. Then
\begin{multline*}A\hm=(x_1-\sum\limits_{j=2}^{|J|}y_j-\alpha)h_{11}\hm+\sum\limits_{i=2}^{k}(x_i-\alpha)h_{i1}\hm+
\sum\limits_{i=k+1}^{|I|}x_ih_{i1}\hm+(x_{|I|+1}\hm-\sum\limits_{j=|J|+2}^{m_2}y_j\hm-\alpha)h_{|I|+1\,|J|+1}+\\
+\sum\limits_{i=|I|+2}^{|I|+n-k}(x_i\hm-\alpha)h_{i\,|J|+1}\hm+
\sum\limits_{i=|I|+n-k+1}^{m_1}x_ih_{i\,|J|+1}\hm+\sum\limits_{j=2}^{|J|}y_jh_{1j}\hm+\sum\limits_{j=|J|+2}^{m_2}y_jh_{|I|+1\,j}\hm+\alpha
f_{1\ldots k\,|I|+1\ldots |I|+n-k}, \end{multline*} where
$\alpha\hm=\frac1k(\sum\limits_{i=1}^{|I|}x_i\hm-\sum\limits_{j=1}^{|J|}y_j)$.
It follows from inequalities \eqref{f6}, that coefficients of this
decomposition are positive. Therefore $A \in \omega(v)$, and $\Pi$
lies in $\omega(v)$. Hence $\Pi$ coincides with $\omega(v)$. This
completes the proof of Theorem \ref{t3}.
\end{proof}

Now let $m_1<n$ and $m_2<n$. In this case the weight cone is
$(m-1)$-dimensional.

\begin{theorem} \label{t4} For the diagonal action of the group $SL(V)$ on the variety
$\mathbb{P}(V)^{m_1}\times \mathbb{P}(V^*)^{m_2}$, where ${m_1,
m_2<n=\dim V}$, the GIT-fan is obtained by cutting of the cone
\begin{equation*}
\Omega=\{(x_1,\ldots,x_{m_1},y_1,\ldots,y_{m_2})\,|\,
x_1+\ldots+x_{m_1}=y_1+\ldots+y_{m_2}; \, x_i, y_j\geqslant0\}
\end{equation*}  by hyperplanes
\begin{equation}\sum_{i \in I}x_i=\sum_{j \in J}y_j, \label{f7} \end{equation} where $I \subset
\{1,\ldots,m_1\},$ $J \subset \{1,\ldots,m_2\},$ $I \ne
\varnothing, \{1,\ldots,m_1\},$ ${J \ne \varnothing,
\{1,\ldots,m_2\}}$.
\end{theorem}

\begin{proof} In this case the weight cone is generated by the weights $\{h_{ij}\}$. It is clear that the weight cone
is contained  in the cone $\Omega$. On the other hand,  edges of
the cone $\Omega$ are precisely the generators of the weight cone.

We need to find all $(m-2)$-dimensional orbit cones. As above let
us construct the graph $\Gamma_v$ for any vector $v$. The
stabilizer $T_{\pi(v)}$ should be of dimension two. Hence the
graph $\Gamma_v$ has two connected components. In this case the
orbit cone $\omega(v)$ is contained in the intersection of the
weight cone $\Omega$ with hyperplane \eqref{f7}, where $I$ and $J$
are sets of numbers: $i\in I$ and $j\in J$ if $v_i$ and $l_j$ lie
in the first connected component of the graph $\Gamma_v$. Finally
it is necessary to prove that intersection of the cone $\Omega$
with hyperplane  \eqref{f7} is the orbit cone associated to some
vector $v$. For this let the vector $v$ be
$(v_1,\ldots,v_{m_1},l_1,\ldots,l_{m_2})$, where $v_i=e_1$,
$l_j=e^1$, if $i \in I, j \in J$ and $v_i=e_2, l_j=e^2$, if $i
\not\in I, j \not\in J$. This completes the proof of Theorem
\ref{t4}.

\end{proof}

\end{document}